\documentclass{amsart}
\usepackage{amssymb,amsfonts}

\newtheorem{theorem}{Theorem}

\newtheorem{corollary}[theorem]{Corollary}

\newtheorem{lemma}[theorem]{Lemma}

\newtheorem{proposition}[theorem]{Proposition}

\begin{document}

\title{Homogeneity degree of some symmetric products}
\author[Hern\'andez-Guti\'errez]{Rodrigo Hern\'{a}ndez-Guti\'{e}rrez}
\author[Mart\'inez-de-la-Vega]{Ver\'{o}nica Mart\'{\i}nez-de-la-Vega}

\address[Hern\'andez-Guti\'errez]{Department of Mathematics and Statistics, University of North Carolina at Charlotte, Charlotte NC 28223.}
\email[Hern\'andez-Guti\'errez]{rodrigo.hdz@gmail.com}

\address[Mart\'inez-de-la-Vega]{Universidad Nacional Aut\'{o}noma de M\'{e}xico, Instituto de Matem\'{a}ticas, Circuito Exterior, Cd. Universitaria, M\'{e}xico, 04510, M\'{e}xico.}
\email[Mart\'inez-de-la-Vega]{vmvm@matem.unam.mx}

\date{\today}


\begin{abstract}
For a metric continuum $X$, we consider the $n^{\tiny\textrm{th}}$-symmetric
product $F_{n}(X)$ defined as the hyperspace of all nonempty subsets of $X$
with at most $n$ points. The homogeneity degree $hd(X)$
of a continuum $X$ is the number of orbits for the action of the group of
homeomorphisms of $X$ onto itself. In this paper we determine $hd(F_{n}(X))$
for every manifold without boundary $X$ and $n\in \mathbb{N}$. We also compute $%
hd(F_{n}[0,1])$ for all $n\in \mathbb{N}$.
\end{abstract}

\maketitle

\section*{Introduction}

A \textit{continuum} is a nonempty compact connected metric space.

Here, the word \textit{manifold} refers to a compact connected manifold with
or without boundary.

Given a continuum $X$, the $n^{\tiny\textrm{th}}$\textit{-symmetric product} of $X$
is the hyperspace
$$
F_{n}(X)=\{A\subset X:A\textrm{ is nonempty and $A$ has at most $n$ points}\}.
$$
The hyperspace $F_{n}(X)$ is considered with the Vietoris topology.

Given a continuum $X$, let $\mathcal{H}(X)$ denote the group of
homeomorphisms of $X$ onto itself. An \textit{orbit} in $X$ is a class of
the equivalence relation in $X$ given by $p$ is equivalent to $q$ if there
exists $h\in \mathcal{H}(X)$ such that $h(p)=q$.

The \textit{homogeneity degree}, $hd(X)$ of the continuum $X$ is
defined as

$$
hd(X)=\textrm{ number of orbits in }X.
$$

When $hd(X)=n$, the continuum $X$ is known to be $\frac{1}{n}$-homogeneous,
and when $hd(X)=1$, $X$ is \textit{homogeneous}.

In \cite{10}, P. Pellicer-Covarrubias studied continua $X$ for which $hd(F_{2}(X))=2$%
. Recently, I. Calder\'{o}n, R. Hern\'{a}ndez-Guti\'{e}rrez and A. Illanes
\cite{2} proved that if $P$ is the pseudo-arc, then $hd(F_{2}(P))=3$. Other papers
on homogeneity degrees of hyperspaces are \cite{4} and \cite{9}.

In this paper we determine $hd(F_{n}(X))$ for every manifold without
boundary $X$ and $n\in \mathbb{N}$. We also compute $hd(F_{n}[0,1])$ for all $%
n\in \mathbb{N}$. Since $F_{1}(X)$ is homeomorphic to $X$, $%
hd(F_{1}(X))=hd(X)$. Thus, $hd(F_{1}(X))=1$ for every manifold without
boundary $X$ and $hd(F_{1}([0,1]))=2$.

\section*{Manifolds without boundary}

We denote by $S^{1}$ the unit circle in the plane.

Given a continuum $X$, $n\in \mathbb{N}$ and subsets $U_{1},\ldots ,U_{m}$
of $X$, let

\begin{center}
$\left\langle U_{1},\ldots ,U_{m}\right\rangle _{n}=\{A\in F_{n}(X):A\subset
U_{1}\cup \ldots \cup U_{m}$ and $A\cap U_{i}\neq \emptyset $ for each $i\in
\{1,\ldots ,m\}\}$.
\end{center}

Then the family $\{\left\langle U_{1},\ldots ,U_{m}\right\rangle _{n}\subset
F_{n}(X):m\leq n$ and $U_{1},\ldots ,U_{m}$ are open in $X\}$ is a basis for
the Vietoris topology in $F_{n}(X)$ \cite{7}. If $A$ is any set and $n\in \mathbb{%
N}$, let $[A]^{n}=\{B\subset A:\left\vert B\right\vert =n\}$. Let $Y$ be a
topological space. A subset $Z\subset Y$ is a \textit{topological type}, or
just a \textit{type}, if for each $h\in \mathcal{H}(Y)$, $h(Z)=Z$.

\begin{lemma}\label{lema1}
Let $X$ be a locally connected continuum such that $%
[X]^{4} $ is a type in $F_{4}(X)$. Then for each $h\in \mathcal{H}(F_{4}(X))$, $%
h([X]^{2})\cap \lbrack X]^{3}=\emptyset $. 
\end{lemma}

\begin{proof}
Suppose to the contrary that there exist $h\in \mathcal{H}%
(F_{4}(X)) $ and $A=\{a_{1},a_{2}\}\in \lbrack X]^{2}$ such that $%
h(A)=B=\{b_{1},b_{2},b_{3}\}\in \lbrack X]^{3}$.

Let $V_{1}$, $V_{2}$ and $V_{3}$ be pairwise disjoint open connected sets of 
$X$ such that $b_{i}\in V_{i}$ for each $i\in \{1,2,3\}$. Let $U_{1},U_{2}$
be disjoint open connected subsets of $X$ such that $a_{1}\in U_{1}$, $%
a_{2}\in U_{2}$ and $h(\left\langle U_{1},U_{2}\right\rangle )\subset
\left\langle V_{1},V_{2},V_{3}\right\rangle $.

The neighborhoods $\left\langle U_{1},U_{2}\right\rangle $ and $\left\langle
V_{1},V_{2},V_{3}\right\rangle $ have different topological structures. We
will use this to arrive to a contradiction.

Let $\mathcal{U}=\left\langle U_{1},U_{2}\right\rangle $ and $\mathcal{V}=h(%
\mathcal{U})$.

First, we will describe the components of $\mathcal{U}\cap \lbrack X]^{4}$.
For each $i\in \{1,2,3\}$, let

\begin{center}
$\mathcal{U}_{i}=\{A\in F_{4}(X):\left\vert A\cap U_{1}\right\vert =i$ and $%
\left\vert A\cap U_{2}\right\vert =4-i\}$.
\end{center}

Then it can be proved that the following properties hold.

\begin{itemize}
 \item[(U1)] $\mathcal{U}_{i}$ is a nonempty subset of $\mathcal{U}\cap \lbrack
X]^{4}$ for each $i\in \{1,2,3\}$,
 \item[(U2)] $\mathcal{U}\cap \lbrack X]^{4}=\mathcal{U}_{1}\cup \mathcal{U}_{2}\cup 
\mathcal{U}_{3}$,
 \item[(U3)] cl$_{F_{4}(X)}(\mathcal{U}_{i})\cap \mathcal{U}_{j}=\emptyset $ if $%
i,j\in \{1,2,3\}$ and $i\neq j$, and
\item[(U4)] $\mathcal{U}_{i}$ is arcwise connected for all $i\in \{1,2,3\}$.
\end{itemize}

Thus, it follows that $\mathcal{U}_{1}$, $\mathcal{U}_{2}$ and $\mathcal{U}%
_{3}$ are exactly the components of $\mathcal{U}\cap \lbrack X]^{4}$.

Now, for each $i\in \{1,2,3\}$, let

\begin{center}
$\mathcal{V}_{i}=\{A\in \mathcal{V}:\left\vert A\cap V%
_{i}\right\vert =2\}$.
\end{center}

The following properties hold.

\begin{itemize}
 \item[(V1)] $\mathcal{V}_{i}$ is a nonempty subset of $\mathcal{V}\cap \lbrack
X]^{4}$ for each $i\in \{1,2,3\}$, 
 \item[(V2)] $\mathcal{V}\cap \lbrack X]^{4}=\mathcal{V}_{1}\cup \mathcal{V}_{2}\cup 
\mathcal{V}_{3}$, and
 \item[(V3)] cl$_{F_{4}(X)}(\mathcal{V}_{i})\cap \mathcal{V}_{j}=\emptyset $ if $%
i,j\in \{1,2,3\}$ and $i\neq j$.
\end{itemize}

Only (V1) requires further explanation. By hypothesis, $h(\mathcal{U}\setminus [X]\sp{4})=\mathcal{V}\setminus [X]\sp{4}$ is non-empty. Choose any  $E\in\mathcal{V}\setminus [X]\sp{4}$, clearly $E\in [X]\sp3$. Now let $i\in\{1,2,3\}$. Since $\mathcal{V}$ is a neighborhood of $E$ in $F_4(X)$, there exists $p_i\in U\setminus E$ such that $E\cup\{p_i\}\in\mathcal{V}$. Then $E\cup\{p_i\}\in\mathcal{V}_i$.

From property (V2), we have $h(\mathcal{U}_{1}\cup \mathcal{U}_{2}\cup 
\mathcal{U}_{3})=\mathcal{V}_{1}\cup \mathcal{V}_{2}\cup \mathcal{V}_{3}$.
Given $i\in \{1,2,3\}$, since $h(\mathcal{U}_{i})$ is connected, there exists $k(i)\in \{1,2,3\}$ such that $h(\mathcal{U}_{i})\subset \mathcal{V}%
_{k(i)}$. Since $h(\mathcal{U}_{1}\cup \mathcal{U}_{2}\cup \mathcal{U}_{3})=%
\mathcal{V}_{1}\cup \mathcal{V}_{2}\cup \mathcal{V}_{3}$, the function $%
k:\{1,2,3\}\rightarrow \{1,2,3\}$, is surjective. Since $\mathcal{V}_{1}$, $%
\mathcal{V}_{2}$ and $\mathcal{V}_{3}$ play symmetric roles, we may assume
that $h(\mathcal{U}_{i})=\mathcal{V}_{i}$ for all $i\in \{1,2,3\}$. Since $%
\mathcal{V}_{1}$, $\mathcal{V}_{2}$ and $\mathcal{V}_{3}$ are pairwise
disjoint, we obtain that indeed, $h(\mathcal{U}_{i})=\mathcal{V}_{i}$ for
all $i\in \{1,2,3\}$. So in fact, $\mathcal{V}_{1}$, $\mathcal{V}_{2}$ and $%
\mathcal{V}_{3}$ are the components of $\mathcal{V}\cap \lbrack X]^{4}$.

Let $C\in \mathcal{U}$ be such that $\left\vert C\cap U_{1}\right\vert =2$
and $\left\vert C\cap U_{2}\right\vert =1$. Denote $D=h(C)$. Since $C\in 
\mathcal{U}-[X]^{4}$ and $[X]^{4}$ is invariant under $h$, we obtain that $%
D\in \mathcal{V}-[X]^{4}$. Then $D\in \left\langle
V_{1},V_{2},V_{3}\right\rangle $, so $\left\vert D\right\vert =3$. So note
that we have the following properties.

(a) There is a neighborhood $\mathcal{R}$ of $C$ such that $\mathcal{R}\cap 
\mathcal{U}_{1}=\emptyset $ and

(b) If $\mathcal{S}$ is a neighborhood of $D$, then $\mathcal{S}\cap 
\mathcal{V}_{i}\neq \emptyset $ for all $i\in \{1,2,3\}$.

Clearly, property (a) contradicts property (b) since $h(\mathcal{U}_{1})=%
\mathcal{V}_{1}$. This contradiction shows that such a homeomorphism $h$ does not
exist. 
\end{proof}

\begin{lemma}\label{lema2}
Let $X$ be an $m$-manifold (with or without boundary) and $1\leq k\leq n$. Suppose that either: (a) $m\geq 2$ and $n\geq 3$; or (b) $%
m=1 $ and $n\geq 4$. Then the set $[X]^{k}$ is a type in $F_{n}(X)$. 
\end{lemma}
\begin{proof}
Let $h\in \mathcal{H}(F_{n}(X))$.

Define $\mathcal{D}_{n}(X)=\{A\in F_{n}(X):A$ has a neighborhood in $%
F_{n}(X) $ that is an $nm$-cell$\}$.

Since the definition of $\mathcal{D}_{n}(X)$ is given in terms of
topological concepts, we have $h(\mathcal{D}_{n}(X))=\mathcal{D}_{n}(X)$.

In the case that $m\geq 2$ and $n\geq 3$, Theorem 17 in \cite{3} implies that $%
h(F_{1}(X))=F_{1}(X)=[X]^{1}$. Moreover, by the first part of the proof of
Theorem 17 in \cite{3}, we have $\mathcal{D}_{n}(X)=[X]^{n}$. Thus, $%
h([X]^{n})=[X]^{n}$. This implies that $h(F_{n-1}(X))=F_{n-1}(X)$. If $3\leq
n-1$ we can repeat the argument to obtain that $h([X]^{n-1})=[X]^{n-1}$.
Proceeding in this way, we have that for each $k\geq 3$, $h([X]^{k})=[X]^{k}$%
. Hence, $h(F_{2}(X))=F_{2}(X)$. Since $h(F_{1}(X))=F_{1}(X)$, we are done.

In the case that $m=1$ and $n\geq 4$, we have that $X$ is either an arc or a
simple closed curve.

By Corollary 4.4 in \cite{6}, $\mathcal{D}_{n}(X)=[X]^{n}$. Thus, $%
h([X]^{n})=[X]^{n}$. Proceeding as before, we obtain that for each $k\geq 4$%
, $h([X]^{k})=[X]^{k}$. Thus, $h(F_{3}(X))=F_{3}(X)$. Moreover, by Corollary
7 in \cite{3}, we have $h(F_{1}(X))=F_{1}(X)$. Finally, Lemma \ref{lema1} implies that $%
h([X]^{2})=[X]^{2}$ and $h([X]^{3})=[X]^{3}$. 
\end{proof}

The following lemma is well known.

\begin{lemma}\label{lema3}
Let $X$ be a manifold without boundary and $n\in \mathbb{N}
$. Then for every $A,B\in \lbrack X]^{n}$, there exists $h\in \mathcal{H}(X)$
such that $h(A)=B$. 
\end{lemma}

\begin{corollary}\label{coro4}
Let $X$ be a manifold without boundary and $n\geq 2$. Then\\ $hd(F_{n}(X))\leq n$. 
\end{corollary}

Applying Lemmas 2 and 3, we conclude the following theorem.

\begin{proposition}\label{teo5}
Let $X$ be an $m$-manifold without boundary with $m\geq
2 $ and let $n\geq 3$. Then $hd(F_{n}(X))=n$. 
\end{proposition}

Given a continuum $X$ and $n\geq 2$, the hyperspace $F_{n}(X)$ is rigid
provided that for each $h\in \mathcal{H}(F_{n}(X))$, $h(F_{1}(X))=F_{1}(X)$.
Rigidity of symmetric products was studied in \cite{3}, where it was shown that
(\cite[Theorem 17]{3}) if $X$ is an $m$-manifold, $m\geq 2$ and $n\geq 3$, then $%
F_{n}(X)$ is rigid. Using a theorem by R. Molski \cite{8}, we show that this
result is also true for $m\geq 3$ and $n=2$.

\begin{lemma}\label{lema6}
Let $X$ be an $m$-manifold with $m\geq 3$. Then $F_{2}(X)$
is rigid. 
\end{lemma}
\begin{proof}
Let $\mathcal{D}_{2}(X)=\{A\in F_{2}(X):A$ has a
neighborhood in $F_{2}(X)$ embeddable in $\mathbb{R}^{2m}\}$. Clearly, each
element $A\in \lbrack X]^{2}$ has a neighborhood homeomorphic to $[0,1]^{2m}$, so $[X]^{2}\subset \mathcal{D}_{2}(X)$.

Given $p\in X$, let $\mathcal{M}$ be a neighborhood of $\{p\}$ in $F_{2}(X)$. Then there exists an open subset $U$ of $X$ such that $p\in U$ and $\left\langle U\right\rangle _{2}\subset \mathcal{M}$. Let $R$ be an $m$-cell
such that $R\subset U$. Then $F_{2}(R)\subset \left\langle U\right\rangle
_{2}\subset \mathcal{M}$. By Theorem 3 in \cite{8}, $F_{2}(R)$ (and then $%
\mathcal{M}$) cannot be embedded in $\mathbb{R}^{2m}$. We have shown that $%
\mathcal{D}_{2}(X)\subset \lbrack X]^{2}$. Therefore, $\mathcal{D}%
_{2}(X)=[X]^{2}$. Since the definition of $\mathcal{D}_{2}(X)$ is given in
topological terms, we conclude that $F_{2}(X)$ is rigid.  
\end{proof}\bigskip

Combining Corollary \ref{coro4} and Lemma \ref{lema6}, we obtain the following result.

\begin{proposition}\label{teo7}
Let $X$ be an $m$-manifold without boundary with $m\geq
3 $. Then $hd(F_{2}(X))=2$. 
\end{proposition}

By the Corollary to Theorem 1 in \cite{8}, if $X$ is a 2-manifold without boundary, then $F_2(X)$ is a 4-manifold without boundary, hence $hd(F_2(X))=1$.

\begin{proposition}\label{teo9}
 Let $n\in \mathbb{N}$. Then $hd(F_{3}(S^{1}))=1$, and if 
$n\neq 3$, then $hd(F_{n}(S^{1}))=n$. 
\end{proposition}
\begin{proof}
By \cite{1}, $F_{3}(S^{1})$ is homeomorphic to the unit sphere in
the Euclidean space $\mathbb{R}^{4}$. So, $hd(F_{3}(S^{1}))=1$. It is well
known that $F_{2}(S^{1})$ is homeomorphic to the Moebious strip (see section
14 in \cite{5}). Thus, $hd(F_{2}(S^{1}))=2$. Clearly, $hd(F_{1}(S^{1}))=1$.
Finally, if $n\geq 4$, by Lemmas \ref{lema2} and \ref{lema3}, $hd(F_n(S\sp1))=n$. 
\end{proof}

We summarize the results of this section in the following theorem.

\begin{theorem}\label{teo10}
Let $X$ be an $m$-manifold without boundary and $n\in\mathbb{N}$. Then
\begin{itemize}
 \item[(a)] if either $m\geq 2$ and $n\neq 2$ or $m=1$ and $n\neq 3$, then $%
hd(F_{n}(X))=n$,
 \item[(b)] if $m=2$ ($X$ is a surface), then $hd(F_{2}(X))=1$, and
 \item[(c)] if $m=1$ ($X$ is a simple closed curve) and $n=3$, then $hd(F_{2}(X))=1$. 
\end{itemize}
\end{theorem}

In the case that $X$ is a manifold with boundary, it seems to be difficult
to give a result so precise as Theorem \ref{teo10}. The following example shows that $%
hd(F_{n}(X))$ could depend not only on $n$ but in the number of components
of the manifold boundary of $X$.

For each $k\in \mathbb{N}$, we can consider a family of disjoint subsets $%
T_{1},T_{2},\ldots ,T_{k}$ of $\mathbb{R}^{3}$, where each $T_{i}$ is a $2$%
-sphere with $i$ handles and if $i\neq j$, then $T_{i}$ is contained in the
unbounded domain of $\mathbb{R}^{3}-T_{j}$. Suppose that $M$ is a closed $3$-ball in $\mathbb{R}^{3}$ containing $T_{1}\cup \ldots \cup T_{k}$ in
its interior. Consider the continuum $X$ which is the closure of the
intermediate region bounded by $M$ and $T_{1}\cup \ldots \cup T_{k}$.
Clearly, $X$ is a manifold with boundary.

By Theorem 17 of \cite{3}, if $n\geq 3$, then $F_{n}(X)$ is rigid. So, if $h\in 
\mathcal{H}(F_{n}(X))$, then $h(F_{1}(X))=F_{1}(X)$. Thus, $%
h|_{F_{1}(X)}:F_{1}(X)\rightarrow F_{1}(X)$ is a homeomorphism. This implies
that $h(T_{i})=T_{i}$. Therefore, for each $n\geq 3$, $hd(F_{n}(X))$ is greater
than $k$.

\section*{The unit interval}

\begin{lemma}\label{lema11}
Let $n\geq 2$ and $A,B\in F_{n}([0,1])$ be such that $%
\left\vert A\right\vert =\left\vert B\right\vert \geq 2$ and $A\cap
\{0,1\}\neq \emptyset \neq B\cap \{0,1\}$. Then there is a homeomorphism $%
h_{0}\in \mathcal{H}(F_{n}([0,1]))$ such that $h(B)=A$. 
\end{lemma}
\begin{proof}
Suppose that $A=\{a_{1},\ldots ,a_{m}\}$ and $%
B=\{b_{1},\ldots ,b_{m}\}$, where $\left\vert A\right\vert =\left\vert
B\right\vert =k\geq 2$, $a_{1}<\ldots <a_{m}$ and $b_{1}<\ldots <b_{m}$. In
the cases:

\begin{itemize}
 \item[(a)] $a_{1}=0=b_{1}$ and $a_{m}=1=b_{m}$,
 \item[(b)] $a_{1}=0=b_{1}$ and $\max \{a_{m},b_{m}\}<1$,
 \item[(c)] $0<\min \{a_{1},b_{1}\}$ and $a_{m}=1=b_{m}$,
 \item[(d)] $0=a_{1}$, $a_{m}<1$, $0<b_{1}$ and $b_{m}=1$, and
 \item[(e)] $0<a_{1}$, $a_{m}=1$, $0=b_{1}$ and $b_{m}<1$,
\end{itemize}
it is easy to show that there is $g\in \mathcal{H}([0,1])$ such that $%
g(B)=A$. Then the induced mapping $h_{0}=F_{n}(g):F_{n}([0,1])\rightarrow
F_{n}([0,1])$ satisfies that $h_{0}(B)=A$. The rest of the cases are similar
to the following one.
\begin{itemize}
 \item[(f)]$a_{1}=0$, $a_{m}<1$, $b_{1}=0$ and $b_{m}=1$. 
\end{itemize}
Thus, we only need to show case (f).

In case (f), for each nonempty closed subset $D$ of $[0,1]$, let $m(D)=\min D
$ and $M(D)=\max D$. Consider the mapping $\varphi :F_{n}([0,1])\rightarrow 
\mathbb{R}^{2}$ given by $\varphi (D)=(m(D),M(D))$. Clearly, $\varphi $ is a
mapping whose image is the triangle $T$ in the plane $\mathbb{R}^{2}$ with
vertices $(0,0)$, $(0,1)$ and $(1,1)$. Let $\Delta $ be the convex segment
in $T$ with end points $(0,0)$ and $(1,1)$. Clearly, there exists $\sigma
=(\sigma _{1},\sigma _{2})\in \mathcal{H}(T)$ such that $\sigma |_{\Delta }$
is the identity in $\Delta $ and $\sigma (0,1)=(0,a_{m})$. Notice that if $%
(x,y)\in T$ and $x<y$, then $\sigma _{1}(x)<\sigma _{2}(y)$.

Given two nondegenerate subintervals $J$ and $K$ of $[0,1]$, let $\psi
(J,K):J\rightarrow K$ be the strictly increasing linear homeomorphism, that
is, for each $t\in J$,

$$
\psi (J,K)(t)=\frac{t-m(J)}{M(J)-m(J)}M(K)+\frac{M(J)-t}{M(J)-m(J)}m(K).
$$

Define $h:F_{n}([0,1])\rightarrow F_{n}([0,1])$ by

$$
h(D)=\left\{ 
\begin{array}{lc}
\psi ([m(D),M(D)],[\sigma _{1}(\varphi (D)),\sigma _{2}(\varphi (D))])(D)\textrm{%
,} & \textrm{if }D\notin F_{1}([0,1])\textrm{,} \\ 
D, & \textrm{if }D\in F_{1}([0,1])\textrm{.}%
\end{array}%
\right.
$$

Clearly, $h$ is continuous in the open set $F_{n}([0,1])-F_{1}([0,1])$. To
complete the proof that $h$ is continuous, let $\{D_{k}\}_{k=1}^{\infty }$
be a sequence in $F_{n}([0,1])-F_{1}([0,1])$ converging to an element $\{p\}\in F_{1}([0,1])$. Then $\lim_{k\rightarrow \infty}M(D_{k})=p=\lim_{k\rightarrow \infty }m(D_{k})$, $\lim_{k\rightarrow \infty}\varphi (D_{k})=(p,p)$ and $\lim_{k\rightarrow \infty }\sigma (\varphi(D_{k}))=(p,p)$. Since for each $k\in \mathbb{N}$, $\psi([m(D_{k}),M(D_{k})],[\sigma _{1}(\varphi (D_{k})),\sigma _{2}(\varphi(D_{k}))])(D_{k})$ is a subset of $[\sigma _{1}(\varphi (D_{k})),\sigma_{2}(\varphi (D_{k}))]$ and $\lim_{k\rightarrow \infty }[\sigma _{1}(\varphi(D_{k})),\sigma _{2}(\varphi (D_{k}))]=\{p\}$, we obtain that 
$$
\lim_{k\rightarrow \infty }\psi ([m(D_{k}),M(D_{k})],[\sigma _{1}(\varphi(D_{k})),\sigma _{2}(\varphi (D_{k}))])(D_{k})=\{p\}.
$$
Thus, $\lim_{k\rightarrow \infty }h(D_{k})=h(\{p\})$. Therefore, $h$ is continuous.

In order to show that $h$ is one-to-one, let $E,D\in F_{n}([0,1])$ be such
that $h(D)=h(E)$.

In the case $D\notin F_{1}([0,1])$, $m(D)<M(D)$, so $\varphi (D)\notin
\Delta $ and $\sigma (\varphi (D))\notin \Delta $. So, $\sigma _{1}(\varphi
(D)<\sigma _{2}(\varphi (D))$. Since $\psi ([m(D),M(D)],[\sigma _{1}(\varphi
(D),\sigma _{2}(\varphi (D))])(m(D))=\sigma _{1}(\varphi (D))$ and $\psi
([m(D),M(D)],[\sigma _{1}(\varphi (D)),\sigma _{2}(\varphi (D))])(M(D))=\sigma
_{2}(\varphi (D)$, we obtain that $m(h(D))=\sigma _{1}(\varphi (D))$, $%
M(h(D))=\sigma _{2}(\varphi (D))$. So, $h(D)$ is nondegenerate. Thus, $h(E)$ is
nondegenerate, $m(h(E))=\sigma _{1}(\varphi (D))$, $M(h(E))=\sigma
_{2}(\varphi (D)$ and $E\notin F_{1}([0,1])$. Similarly, $m(h(E))=\sigma
_{1}(\varphi (E))$ and $M(h(E))=\sigma _{2}(\varphi (E))$. Hence, $\sigma
(\varphi (E))=\sigma (\varphi (D))$, $\varphi (E)=\varphi (D)$, $m(E)=m(D)$
and $M(E)=M(D)$. Thus,

$$
\begin{array}{c}
\psi ([m(D),M(D)],[\sigma _{1}(\varphi (D)),\sigma _{2}(\varphi (D))])=\\
\psi ([m(E),M(E)],[\sigma _{1}(\varphi (E),\sigma _{2}(\varphi (E))])
\end{array}
$$

Since this mapping is one-to-one, we conclude that $E=D$.

In the case that $D\in F_{1}([0,1])$, proceeding as before, we obtain that $%
E\in F_{1}([0,1])$. Thus, $D=h(D)=h(E)=E$. Therefore, $h$ is one-to-one.

Now we check that $h$ is onto. Take $E\in F_{n}([0,1])$. If $E\in
F_{1}([0,1])$, then $E=h(E)$ and we are done. So, we suppose that $E\notin
F_{1}([0,1])$. Let $x=m(E)$, $y=M(E)$ and $(a,b)=\sigma ^{-1}(x,y)$. Then $%
x<y$ and $a<b$. Since $\psi ([a,b],[x,y]):[a,b]\rightarrow \lbrack x,y]$ is
onto and $E\subset \lbrack x.y]$, we can define $D=(\psi
([a,b],[x,y]))^{-1}(E)\in F_{n}([0,1])$. Note that $D\notin F_{1}([0,1])$
and $h(D)=E$. Therefore, $h$ is onto.

We have shown that $h$ is a homeomorphism.

Note that $h(B)=\psi ([0,1],[\sigma _{1}(0,1),\sigma _{2}(0,1)])(B)\subset
\lbrack \sigma _{1}(0,1),\sigma _{2}(0,1)]=[0,a_{m}]$ and $%
\{0,a_{m}\}\subset h(B)$. Thus, $A$ and $h(B)$ are as in (b). So, there
exists a homeomorphism $h_{1}\in \mathcal{H}(F_{n}([0,1]))$ such that $%
h_{1}(h(B))=A$. Therefore, $h_{0}=h_{1}\circ h$ is a homeomorphism such that 
$h_{0}(B)=A$. 
\end{proof}

\begin{theorem}\label{teo12}
Let $n\in \mathbb{N}$. Then:
\begin{itemize}
 \item[(a)]if $n\notin \{2,3\}$, then $hd(F_{n}([0,1]))=2n$, and
 \item[(b)]if $n\in \{2,3\}$, then $hd(F_{n}([0,1]))=2$. 
\end{itemize}
\end{theorem}

\begin{proof}
(a) Since $F_{1}([0,1])$ is homeomorphic to $[0,1]$, $%
hd(F_{1}([0,1]))=2$. Let $n\geq 4$, $1\leq k\leq n$ and $h\in \mathcal{H(}%
F_{n}([0,1]))$. By Lemma \ref{lema2}, $h([[0,1]]^{k})=[[0,1]]^{k}$. Note that $%
[[0,1]]^{k}=\mathcal{D}_{1}\cup \mathcal{D}_{2}$, where $\mathcal{D}%
_{1}=\{A\in \lbrack \lbrack 0,1]]^{k}:A$ has a neighborhood $\mathcal{M}$ in 
$[[0,1]]^{k}$ that is a $k$-cell and $A$ is in the manifold boundary of $%
\mathcal{M}\}$ and $\mathcal{D}_{2}=\{A\in \lbrack \lbrack 0,1]]^{k}:A$ has
a neighborhood $\mathcal{M}$ in $[[0,1]]^{k}$ that is a $k$-cell and $A$ is
not in the manifold boundary of $\mathcal{M}\}$. Clearly, $h(\mathcal{D}%
_{1})=\mathcal{D}_{1}$ and $h(\mathcal{D}_{2})=\mathcal{D}_{2}$. By Lemma
\ref{lema11}, $\mathcal{D}_{2}$ is an orbit in $F_{n}([0,1])$. Clearly, $\mathcal{D}%
_{1}$ is an orbit in $F_{n}([0,1])$. Therefore, $F_{n}([0,1])$ contains
exactly $2n$ orbits.

By sections 13 and 14 in \cite{5}, $F_{2}([0,1])$ is a $2$-cell and $F_{3}([0,1])$
is a $3$-cell. This implies (b).
\end{proof}

\bigskip

\section*{Acknowledgements}

The research in this paper was carried out during the 9th Research Workshop
in Hyperspaces and Continuum Theory held in the cities of Toluca and M\'{e}%
xico, during June, 2015. The authors would like to thank Irving Calder\'{o}%
n, Vianey C\'{o}rdoba, Alejandro Illanes, Jorge Mart\'{\i}nez-Montejano,
Yaziel Pacheco, Erick Rodr\'{\i}guez and Claudia Solis who joined the
discussion on degree of homogeneity in symmetric products during the
workshop.

\bigskip

This paper was partially supported by the projects ``Teor\'{\i}a de continuos
e hiperespacios (0221413)'' of Consejo Nacional de Ciencia y Tecnolog\'{\i}a
(CONACyT), 2013 and ``Teor\'{\i}a de Continuos, Hiperespacios y Sistemas Din%
\'{a}micos'' (IN104613) of PAPIIT, DGAPA, UNAM. The research of Hern\'{a}%
ndez-Guti\'{e}rrez was also partially funded by the CONACyT posdoctoral
program, registered with proposal number 234735. He would also like to thank
the University of North Carolina at Charlotte for providing office space.

\end{document}